\documentclass[reqno, 12pt]{amsart}

\makeatletter
\let\origsection=\section \def\section{\@ifstar{\origsection*}{\mysection}} 
\def\mysection{\@startsection{section}{1}\z@{.7\linespacing\@plus\linespacing}{.5\linespacing}{\normalfont\scshape\centering\S}}
\makeatother    

\usepackage{graphicx}
\usepackage{geometry}
\numberwithin{equation}{section}
\numberwithin{figure}{section}

\usepackage{xcolor}
\usepackage{hyperref}
\hypersetup{
    unicode,
    colorlinks,
    linkcolor={red!60!black},
    citecolor={green!60!black},
    urlcolor={blue!60!black}
}

\usepackage{amsfonts,amsthm,amssymb,amsmath}
\usepackage{graphicx}
\usepackage{tikz}
\usepackage[usenames,dvipsnames]{pstricks}
\usepackage{epsfig}
\usepackage{enumitem}
\setlist[enumerate]{label=(\arabic*), ref=(\arabic*)}
\usepackage{verbatim}
\usepackage{soul}
\usepackage{thmtools}
\usepackage[capitalise, nameinlink]{cleveref}
\usepackage{mathtools}
\usepackage[a]{esvect}
\usepackage[utf8]{inputenc}
\usepackage[T1]{fontenc}
\usepackage{lmodern}
\usepackage[babel]{microtype}

\let\polishlcross=\l
\def\l{\ifmmode\ell\else\polishlcross\fi}
\let\emptyset=\varnothing

\let\theta=\vartheta
\let\rho=\varrho
\let\phi=\varphi

\newcommand{\set}[2]{{\{ #1 \ | \ #2 \}}}

\makeatletter
\newcommand{\nameditemlabel}[2]{\item[#1]\begingroup\def\@currentlabel{#1}\phantomsection\label{#2}\endgroup}
\newcommand{\itemlabel}[1]{\item\begingroup\def\@currentlabel{\labelenumi}\phantomsection\label{#1}\endgroup}
\makeatother

\renewcommand{\labelenumi}{\textit{(\roman{enumi})}}

\theoremstyle{plain}
\newtheorem{theorem}{Theorem}[section]

\newtheorem{corollary}[theorem]{Corollary}
\newtheorem{lemma}[theorem]{Lemma}
\newtheorem{claim}[theorem]{Claim}

\theoremstyle{definition}
\newtheorem{definition}[theorem]{Definition}
\newtheorem{remark}[theorem]{Remark}

\newtheoremstyle{note}{5pt}{5pt}{}{}{\itshape}{:}{ }{}
\theoremstyle{note}

\newtheoremstyle{case}{5pt}{5pt}{}{}{\bfseries}{:}{ }{}
\theoremstyle{case}
\newtheorem{case}{Case}

\usepackage{etoolbox}
\AtBeginEnvironment{proof}{\setcounter{step}{0}}
\AtBeginEnvironment{proof}{\setcounter{case}{0}}
\AtBeginEnvironment{case}{\setcounter{subcase}{0}}

\newcommand{\N}{\mathbb{N}}

\newcommand{\A}{\mathcal{A}}
\newcommand{\C}{\mathcal{C}}
\newcommand{\D}{\mathcal{D}}
\newcommand{\V}{\mathcal{V}}

\newcommand{\PP}{\mathcal{P}}
\newcommand{\NN}{\mathcal{N}}

\newcommand{\singleset}[1]{{\{#1\}}}

\newcommand{\then}{\ \Rightarrow \ }

\newcommand{\cupdot}{\mathbin{\dot{\cup}}}
\newcommand{\minset}[1]{\min\{#1\}}

\newcommand{\pairs}[1]{[#1]^2}
\newcommand{\range}[2]{\singleset{#1, \ldots, #2}}
\newcommand{\family}[2]{\singleset{#1}_{#2}}
\newcommand{\floor}[1]{\left\lfloor #1 \right\rfloor}

\newcommand{\xy}{\singleset{x, y}}
\newcommand{\Cless}{\C_{< \infty}}
\newcommand{\Cinf}{\C_\infty}
\newcommand{\alambda}{\vv{\lambda}}

\title{The strong Nash-Williams orientation theorem for rayless graphs}
\author{Max Pitz}
\author{Jacob Stegemann}
\address{Universit\"at Hamburg, Department of Mathematics, Bundesstrasse 55 (Geomatikum), 20146 Hamburg, Germany}
\email{\{max.pitz, jacob.stegemann\}@uni-hamburg.de}

\begin{document}

\begin{abstract}
    In 1960, Nash-Williams proved his strong orientation theorem that every finite graph has an orientation in which the number of directed paths between any two vertices is at least half the number of undirected paths between them (rounded down).

    Nash-Williams conjectured that it is possible to find such orientations for infinite graphs as well.
    We provide a partial answer by proving that all rayless graphs have such an orientation.
\end{abstract}

\maketitle

\section{Introduction}

For two vertices $x,y$ in a multi-graph $G$ let $\lambda(x,y)$ denote the maximum number of edge-disjoint paths from $x$ to $y$,
and for two vertices $x,y$ in an oriented multi-graph $\vv{G}$ let $\alambda(x,y)$ denote the maximum number of arc-disjoint directed paths from $x$ to $y$.
An orientation $\vv{G}$ of a multigraph $G$ is \textbf{well-balanced} if
\[
    \alambda(x,y) \geq \left\lfloor \frac{\lambda(x,y)}{2} \right\rfloor
\]
for any two distinct vertices $x,y \in G$.

The \emph{Strong orientation theorem} of Nash-Williams from 1960 says that every finite multigraph admits a well-balanced orientation \cite{NashWilliams1960}.
In the same paper, Nash-Williams claimed that his result also holds for infinite graphs as well -- but 10 years later, Nash-Williams retracted his claim in \cite[\S~8]{nash1969well}, and it has remained open ever since whether the strong orientation theorem holds for infinite graphs as well. 

The strong orientation theorem implies in particular that every finite $2k$-edge-connected multigraph has a $k$-arc-connected orientation. 
Regarding this global version for infinite graphs, a breakthrough result by Thomassen from 2015 asserts that every (infinite) $8k$-edge-connected multigraph has a $k$-arc-connected orientation \cite{thomassen2016orientations}.
This result was recently improved to a bound of $4k$, and the optimal $2k$ bound for locally finite graphs with countably many ends \cite{assem2025nash}.
However, establishing the optimal $2k$ bound for \emph{all} graphs remains a challenging open problem.

In contrast, no such affirmative results for infinite graphs seem to be known for the original strong orientation problem.
The purpose of the present paper is to improve this state of affairs by establishing the strong Nash-Williams orientation conjecture for all rayless graphs.

\begin{restatable}{theorem}{thmraylessintro}\label{thm:rayless-intro}
    Every rayless graph admits a well-balanced orientation.
\end{restatable}

That the rayless case is often more tractable than the general case is not unheard of:
See \cite{bruhn2010every} for a proof of the unfriendly partition conjecture for rayless graphs, and \cite{aharoni1983menger,polat1983extension} for proofs of the  Erd\H{o}s-Menger conjecture for rayless graphs, more than 30 years before its eventual solution due to Aharoni and Berger in \cite{aharoni2009menger}. 

As an important step towards \cref{thm:rayless-intro}, we reduce the strong orientation problem to the countable case, thereby extending the corresponding result of Thomassen in the case of global connectivity \cite[\S8]{thomassen2016orientations}.
This is done via the following reduction result that also caters to the rayless case.

\begin{restatable}{theorem}{reductionIntro}\label{thm:reduction-to-countable-case}
    Let $\A$ be a class of graphs that is closed under subgraphs.
    If every countable graph in $\A$ has a well-balanced orientation, then all graphs in $\A$ have a well-balanced orientation.
\end{restatable}

By letting $\A$ be the class of all [rayless] graphs, we see that it suffices to prove the strong orientation conjecture for countable [rayless] graphs.

\section{Preliminaries}\label{section:preparation}

We denote with $\N = \{0,1,2,3,\ldots\}$ the set of natural numbers including $0$.
For $n \in \N$ we write $[n] \coloneq \range{0}{n}$.
In this paper, graphs may have multiple parallel edges but no loops.
A graph in the classical sense (i.e.\ without multiple edges) will be called a ``simple graph''.
Formally, we consider graphs defined as follows.

A \textbf{graph} $G$ is a triple $(V, E, f)$, where $V$ and $E$ are sets and $f$ is a map
\[
    f \colon E \to \pairs{V}.
\]
The set $V(G) \coloneq V$ is called the \textbf{vertices}, $E(G) \coloneq E$ the \textbf{edges} and $f$ the \textbf{incidence map} of $G$.
For two distinct vertices $x, y \in V(G)$ we write $E_G(x, y) \coloneq f^{-1}(\xy)$ for the set of edges between $x$ and $y$.
For two subsets $X, Y \subseteq V(G)$ we write $E_G(X, Y)$ for the set of edges with one endvertex in $X$ and the other in $Y$.
An \textbf{oriented graph} $G$ is a triple $(V, E, f)$, where
\[
    f \colon E \to \set{(x, y) \in V^2}{x \neq y}.
\]

For a graph $G$ write $\C(G)$ for the set of its components.
If $H \subseteq G$ and $C \in \C(G - H)$ is a component, write $N(C)$ for the neighbors of $C$ in $H$.
For two distinct vertices $x, y \in V(G)$ we write $\lambda_G(x, y)$ for the cardinality of the edge-connectivity between $x$ and $y$ in $G$.

We call an orientated graph $D$ with incidence map $f_D$ an \textbf{orientation} of a graph $G$ with incidence map $f_G$ if $V(D) = V(G)$, $E(D) = E(G)$ and $f_G(e) = xy$\footnote{
    This is a shorthand notation for the set $\xy$.
} if and only if $f_D(e) = (x, y)$ or $f_D(e) = (y, x)$ for all edges $e \in E(G)$ and vertices $x, y \in V(G)$.     In the following we will only ever consider one distinguished orientation $D$ for a graph $G$.
Therefore, we will always write $\vv{G}$ for $D$ and also $\alambda_G$ for $\lambda_D$.

All other terminology for graphs is as in \cite{Diestel}.

We now introduce the main definition of this paper:

\begin{definition}
    An orientation $\vv{G}$ of a graph $G$ is \textbf{well-balanced} if
    \[
        \alambda_G(x, y) \geq \begin{cases}
            \floor{\frac{\lambda_G(x, y)}{2}} & \text{if } \lambda_G(x, y) < \infty,\\
            \lambda_G(x, y) & \text{else.}
        \end{cases}
    \]
    for all distinct vertices $x, y \in V(G)$.
\end{definition}

So in particular, if $\lambda_G(x,y)$ is an infinite cardinal such as $\aleph_0,\aleph_1,\ldots$, then there should be as many oriented paths in $\vv{G}$ as there were unoriented paths in $G$, as intuitively, dividing by 2 makes no difference to an infinite set. 

Note that for countable graphs, no such distinction into different infinite cardinalities is necessary, and an oriented graph is well-balanced if 
 \[
        \alambda_G(x, y) \geq \begin{cases}
            \floor{\frac{\lambda_G(x, y)}{2}} & \text{if } \lambda_G(x, y) < \infty,\\
            \infty & \text{else.}
        \end{cases}
    \]
 for all distinct vertices $x, y \in V(G)$.

\section{Well-balanced orientations and tree-decompositions}

We begin by proving a result essentially saying that it suffices to prove the orientation conjecture for $2$-connected graphs, as we can combine well-balanced orientations of the individual blocks to obtain a well-balanced orientation of the whole graph.

For technical reasons, we state this observation in terms of tree-decompositions of adhesion $\leq 1$ as follows.

\begin{definition}
    Let $G$ be a graph with incidence map $f$.
    A pair $(T, \V)$, where $T$ is a tree and $\V = \family{V_t}{t \in V(T)}$ a collection of subsets $V_t \subseteq V(G)$ for all $t \in V(T)$ is called a \textbf{tree-decomposition} of $G$ if
    \begin{itemize}[align=left, leftmargin=1em]
        \nameditemlabel{(T1)}{it:tree-decomposition-t1} $V(G) = \bigcup_{t \in V(T)} V_t$,
        \nameditemlabel{(T2)}{it:tree-decomposition-t2} $\forall e \in E(G) \ \exists t \in V(T) \colon f(e) \subseteq V_t$ and
        \nameditemlabel{(T3)}{it:tree-decomposition-t3} $\forall t_1, t_2, t_3 \in V(T) \colon t_2 \in V(t_1 T t_3) \then V_{t_1} \cap V_{t_3} \subseteq V_{t_2}$.
    \end{itemize}
    The induced subgraphs $G[V_t]$ for $t \in V(T)$ are called the \textbf{parts} of the tree-decomposition,
    and the sets $V_{t_1} \cap V_{t_2}$ for edges $t_1 t_2 \in E(T)$ are its \textbf{adhesion sets}.
    The tree-decomposition has \textbf{adhesion $\leq k$} for $k \in \N$ if all adhesion sets have size $\leq k$ elements.
\end{definition}

The various parts of a tree-decomposition are in general not edge-disjoint,  but in the following special case, they are.

\begin{lemma}
    Let $G$ be a graph and $(T, \V)$ with $\V = \family{V_t}{t \in V(T)}$ be a tree-decomposition of adhesion $\leq 1$ of $G$.
    Then the parts of $(T, \V)$ induce a partition of the edge set of $G$.
\end{lemma}

\begin{proof}
    By \ref{it:tree-decomposition-t1} and \ref{it:tree-decomposition-t2} the union of all parts is $G$.
    We are left to show that the parts are edge disjoint.
    Suppose there exists an edge $e \in E(G)$, such that $f(e) \subseteq V_{t_1}$ and $f(e) \subseteq V_{t_3}$ for distinct $t_1, t_3 \in V(T)$.
    Then every $V_{t_2}$ for $t_2 \in V(t_1 T t_3)$ contains the vertices $f(e)$ by \ref{it:tree-decomposition-t3}.
    Therefore, every adhesion set corresponding to edges on the path $t_1 T t_3$ also contains $f(e)$ (there exists an edge since $t_1 \neq t_3$).
    But $|f(e)| = 2$ is greater than $1$.
\end{proof}

The previous lemma allows us to use the parts of a tree-decomposition of adhesion $\leq 1$ to ``glue'' orientations together as follows.

\begin{lemma}\label{lem:well-balanced-cut-decomposition}
    Let $G$ be a graph and $(T, \V)$ with $\V = \family{V_t}{t \in V(T)}$ be a tree-decomposition of adhesion $\leq 1$ of $G$.
    Assume that for all $t \in V(T)$ there exists a well-balanced orientation $\vv{G[V_t]}$.
    Then $\vv{G} \coloneq \bigcup_{t \in V(T)} \vv{G[V_t]}$ is a well-balanced orientation of $G$.
\end{lemma}

\begin{proof}
    For $t \in V(T)$ we write $G_t$ for the part $G[V_t]$.
    Let $x, y \in V(G)$ be two distinct vertices.
    We want to show that $\vv{G}$ is  well-balanced between $x$ and $y$.
    If $\lambda_G(x, y) = 0$, then this is trivially satisfied, so we may assume that $x$ and $y$ are in the same component of $G$.
    We choose $t_x, t_y \in V(T)$ such that $x \in V_{t_x}$ and $y \in V_{t_y}$ and the distance $d \in \N$ of $t_x$ and $t_y$ in $T$ is minimal.

    We prove that $\vv{G}$ is  well-balanced between $x$ and $y$ by induction on $d$.
    In the induction start we have $d = 0$ and thus $G_{t_x} = G_{t_y}$.
    Any $x-y$-path in $G$ is contained in $G_{t_x}$, since $(T, \V)$ is a tree-decomposition of adhesion $\leq 1$.
    Since $\vv{G_{t_x}}$ is  well-balanced, we are done.

    For the induction step let $d > 0$.
    Then in particular $t_x \neq t_y$ and the unique path
    from $t_x$ to $t_y$ in $T$ has length $d$.
    This path starts with an edge $t_x t \in E(T)$.
    Since $\lambda(x,y) > 0$, it follows from \cite[Lemma~12.3.1]{Diestel} that the corresponding adhesion set $V_{t_x} \cap V_t$ is non-empty and therefore contains a vertex $x' \in V_{t_x} \cap V_t$.
    
    By the minimality of $d$ we have $x \notin V_t$ and $y \notin V_{t_x}$, so $x' \neq x, y$.
    Since the adhesion sets have size $\leq 1$, the vertex $x'$ is a cutvertex separating $x$ from $y$ in $G$.
    Therefore,
 
  $$\lambda_G(x, y) = \minset{\lambda_G(x, x'), \lambda_G(x', y)} \; \text{ and }   \;    \alambda_G(x, y) = \minset{\alambda_G(x, x'), \alambda_G(x', y)}.$$

    By induction, $\vv{G}$ is  well-balanced between $x$ and $x'$ and  well-balanced between $x'$ and $y$.
    In the case, where both $\lambda_G(x, x')$ and $\lambda_G(x', y)$ are finite, we have
    \begin{align*}
        \alambda_G(x, y)
        &= \minset{\alambda_G(x, x'), \alambda_G(x', y)}\\
        &\geq \minset{\floor{\frac{\lambda_G(x, x')}{2}}, \floor{\frac{\lambda_G(x', y)}{2}}}\\
        &= \floor{\frac{\minset{\lambda_G(x, x'), \lambda_G(x', y)}}{2}} =  \floor{\frac{\lambda_G(x, y)}{2}}.
    \end{align*}
    If only one of them is infinite, we argue similarly, and if both are infinite, we have
    \[
        \alambda_G(x, y) = \minset{\alambda_G(x, x'), \alambda_G(x', y)} =  \minset{\lambda_G(x, x'), \lambda_G(x', y)} =  \lambda_G(x, y).\qedhere
    \]
\end{proof}

\section{Contractions in countable graphs}\label{section:contractions}

In our main result in \cref{section:main-theorem}, we will prove by induction that every countable, rayless graph has a well-balanced orientation.
We will apply the induction assumption to certain contracted subgraphs, and then lift any well-balanced orientation to the uncontracted graph.
This section will make this procedure formal.
Note that we will work throughout this section with a countable graph $G$, for which the definition of a well-balanced orientation $\vv{G}$ simplifies to  
\[
    \alambda_G(x, y) \geq \begin{cases}
        \floor{\frac{\lambda_G(x, y)}{2}} & \text{if } \lambda_G(x, y) < \infty,\\
        \infty  & \text{if } \lambda_G(x, y) = \infty.
    \end{cases}
\]

\subsection{Contractions of graphs preserving the edges}

\begin{definition}
    Let $\sim$ be an equivalence relation on a set $X$.
    If for some $x \in X$ the set $\singleset{x}$ is an equivalence class under $\sim$, we may write $\dot{x}$ for $\singleset{x}$.
    We say that the equivalence relation $\sim$ is \textbf{generated by} pairwise disjoint subsets $\family{X_i}{i \in I}$ of $X$ if for all $x, y \in X$
    \[
        x \sim y \iff x = y \text{ or } \exists i \in I \colon x, y \in X_i.
    \]
    If $X_i \neq \emptyset$ for $i \in I$, then $X_i$ is an equivalence class and hence an element of $X / {\sim}$.
    All other elements of $X / {\sim}$ are of the form $\dot{x}$ for $x \in X - \bigcup_{i \in I} X_i$.
    Hence, we have
    \[
        X / {\sim} = \set{X_i}{i \in I, X_i \neq \emptyset} \cup \set{\dot{x}}{x \in X - \bigcup_{i \in I} X_i}.
    \]
\end{definition}

\begin{definition}
    Let $G$ be a graph.
    We call an equivalence relation $\sim$ on $V(G)$ \textbf{contractible in $G$} if for all $x, y \in V(G)$, $x \sim y$ implies that $x$ and $y$ are independent.

    In this case we write $\pi \colon V(G) \to V(G) / {\sim}$ for the map that sends each element to its equivalence class.
    Let $f$ be the incidence map of $G$, then we define the \textbf{contraction of $G$ by $\sim$} as the following graph
    \begin{align*}
        V(G / {\sim}) &\coloneq V(G) / {\sim},\\
        E(G / {\sim}) &\coloneq E(G),
    \end{align*}
    together with the incidence map
    \[
        f^{\sim} \colon E(G / {\sim}) \to \pairs{V(G / {\sim})}, e \mapsto \pi(f(e)).
    \]
Since equivalent vertices are independent, edges do not get contracted to loops and so $f^{\sim}$ is well defined.

    We call a subset $N \subseteq V(G)$ \textbf{contractible in $G$} if it is a (possibly empty) set of independent vertices.
    We call $G / N \coloneq G / {\sim}$ the \textbf{contraction of $G$ by $N$}, where $\sim$ is the equivalence relation generated by $N$.
\end{definition}

    Intuitively $G / {\sim}$ is obtained from $G$ by contracting each equivalence class of $\sim$ to a single vertex and keeping all edges.     Also if $N \neq \emptyset$, then
    \[
        V(G / N) = \singleset{N} \cup \set{\dot{a}}{a \in V(G) - N}.
    \]

\begin{definition}
    Let $G$ be a graph, $\sim$ be contractible in $G$ and $H \subseteq G$ be a subgraph.
    Let $f$ be the incidence map of $G$.
    We write $\pi(H) \subseteq G / {\sim}$ for the subgraph defined by
    \begin{align*}
        V(\pi(H)) &\coloneq \pi(V(H))\\
        E(\pi(H)) &\coloneq E(H)\\
        f_{\pi(H)} &\coloneq f^{\sim}|_{E(H)}.
    \end{align*}
    We say that \textbf{$H$ induces $\pi(H) \subseteq G / {\sim}$}.
\end{definition}

\begin{definition}
    Let $G$ be a graph, $\sim$ be contractible in $G$ and $H \subseteq G / {\sim}$ be a subgraph.
    Let $f$ be the incidence map of $G$.
    We say that \textbf{$H$ lifts to $H'$}, if $H' \subseteq G$ is the following subgraph.
    \begin{align*}
        V(H') &= \bigcup f(E(H)) = \set{v \in V(G)}{\exists e \in E(H) \colon v \in f(e)}\\
        E(H') &= E(H)\\
        f_{H'} &= f|_{E(H)}.
    \end{align*}
\end{definition}

In other words, the vertex set of the lift is defined to be minimal with the property that it supports the lifted edges.

\begin{lemma}\label{lem:connected}
    Let $G$ be a graph, $\sim$ be contractible in $G$ and $H \subseteq G$ be a connected subgraph.
    Then $\pi(H)$ is connected as well.
\end{lemma}

\begin{proof}
 This follows readily from the observation that if
$$P = x_0e_1x_1e_1\ldots e_kx_{k}$$
is a path in $G$ from $x_0$ to $x_k$, then 
$$\pi(P) = \pi(x_0)e_1\pi(x_1)e_1\ldots e_k \pi(x_{k})$$
 is a walk in $\pi(H)$ from $\pi(x_0)$ to $\pi(x_k)$.
\end{proof}

\subsection{Lifting well-balanced orientations from contractions}

 Whenever $\vv{G / {\sim}}$ is oriented, this orientation induces a ``compatible'' orientation on $G$ (and vice versa), since their edge sets are the same.
    Formally for all $e \in E(G)$ and $xy = f(e)$ we define $\vv{f}(e) \coloneq (x, y)$ if $\vv{f^{\sim}}(e) = (\pi(x), \pi(y))$.
    Therefore, the ``inducing'' and ``lifting'' of subgraphs from the previous definitions can also be performed for oriented $G$ and $G / {\sim}$ (if their orientations are ``compatible'' in this sense).

The following technical results isolate some connectivity properties that are preserved when lifting a well-balanced orientations from $G / {\sim}$ to $G$.

\begin{definition}
    Let $G$ be a  graph.
    We call a subset $K \subseteq \pairs{V(G)}$ a \textbf{skeleton in $G$} and write
    \[
        ||K|| \coloneq f^{-1}(K) = \bigcup_{xy \in K} E_G(x, y) \subseteq E(G)
    \]
    for the set of edges in $G$ spanned by the vertex pairs in $K$.
\end{definition}

\begin{lemma}\label{thm:1}
    Let $G$ be a countable graph, $K$ be a skeleton in $G$ and $\sim$ be contractible in $H \coloneq G - ||K||$ such that
    \begin{equation}\label{eq:K}
        \forall xy \in K \colon x \sim y \ \vee \ \lambda_H(x, y) = \infty.
    \end{equation}
    Then for all distinct vertices $x, y \in V(G)$ with $\pi(x) \neq \pi(y)$ we have
    \[
        \lambda_{H / {\sim}}(\pi(x), \pi(y)) \geq \lambda_G(x, y).
    \]
\end{lemma}

\begin{proof}
    We start with the following claim, which we later use to recursively find replacements for paths.

    \begin{claim}
        Let $P'$ be a path in $G$.
        Let $x, y \in V(G)$ be its start- and endvertex respectively and let $L \subseteq E(H / {\sim})$ be some finite set of edges, edge-disjoint from $P'$.
        Then there exists a path $P$ from $\pi(x)$ to $\pi(y)$ in $H / {\sim}$ that is edge-disjoint from $L$.
    \end{claim}

    \begin{proof}[Proof of claim]\renewcommand{\qedsymbol}{$\Diamond$}
        We prove the claim by induction on the cardinality of $E(P') \cap ||K||$.
        In the induction start this intersection is empty, hence $E(P') \subseteq E(H)$.
        Therefore, $\pi(P') \subseteq H / {\sim}$ is connected by \cref{lem:connected}, so we let $P \subseteq \pi(P')$ be a path from $\pi(x)$ to $\pi(y)$ in $H / {\sim}$.
        Now $P$ is edge-disjoint from $L$, since $P'$ is.

        In the induction step we perform the following procedure to some edge $e \in E(P') \cap ||K||$.
        Let $a, b \in V(P')$ be the vertices incident with $e$ such that $a$ precedes $b$ along the order of $P'$.
        Since $e \in ||K||$, we have $ab \in K$ and therefore \cref{eq:K} implies $a \sim b$ or $\lambda_H(a, b) = \infty$.
        
        If $a \sim b$, then we apply the induction hypothesis to $P' a$ and $b P'$ individually (for the same $L$).
        The resulting paths will be from $\pi(x)$ to $\pi(a)$ and $\pi(b)$ to $\pi(y)$ in $H / {\sim}$ respectively.
        Since $a \sim b$ implies $\pi(a) = \pi(b)$, the union of these paths contains a path as desired.
        
        If on the other hand $\lambda_H(a, b) = \infty$, then we replace $e$ in $P'$ by a path $R$ from $a$ to $b$ in $H$ edge-disjoint from $L$.
        Then we apply the induction hypothesis to the resulting path.
    \end{proof}
    
    Let $x, y \in V(G)$ be two distinct vertices with $\pi(x) \neq \pi(y)$.
    Let $l \in \N$ with $l \leq \lambda_G(x, y)$.
    To show the assertion of the theorem it suffices to show that $\lambda_{H / {\sim}}(\pi(x), \pi(y)) \geq l$.

    Let $\PP' \coloneq \range{P'_1}{P'_l}$ be a set of pairwise edge-disjoint paths from $x$ to $y$ in $G$.
    We will recursively define paths $P_1, \dots, P_l$ from $\pi(x)$ to $\pi(y)$ in $H / {\sim}$, such that $P_i$ is edge-disjoint from all paths in $\singleset{P_1, \dots, P_{i - 1}} \cup (\PP' \setminus \singleset{P'_i})$ for $i = 1, \dots, l$.

    Assume that the paths $P_1, \dots, P_n$ are already defined for some $0 \leq n < l$.
    Consider the path $P'_{n + 1}$ from $x$ to $y$ in $G$, which is edge-disjoint from all paths in $\PP \coloneq \singleset{P_1, \dots, P_n} \cup (\PP' \setminus \singleset{P'_{n + 1}})$.\footnote{
        It is edge-disjoint from all paths in $\singleset{P_1, \dots, P_n}$, by the property of the paths $P_i$ for $i = 1, \dots, n$.
        It is also edge-disjoint from all paths in $\PP' \setminus \singleset{P'_{n + 1}}$, since the set $\PP'$ was chosen to be a set of pairwise edge-disjoint paths.
    }

    Using the claim, there exists a path $P_{n + 1}$ from $\pi(x)$ to $\pi(y)$ in $H / {\sim}$, wich is still edge-disjoint from all paths in $\PP$.
\end{proof}

\begin{lemma}\label{thm:2}
    Let $G$ be a countable graph, $K$ be a skeleton in $G$, $\sim$ be contractible in $H \coloneq G - ||K||$ and $\vv{H}$ be an orientation such that
    \begin{equation}\label{eq:sim}
        \forall xy \in \pairs{V(G)} \colon x \sim y \then (\alambda_H(x, y) = \infty \ \vee \ |E_G(x, y)| = \infty).
    \end{equation}
    Then there exists an orientation $\vv{G}$ such that $\vv{G} - ||K|| = \vv{H}$ and for all distinct vertices $x, y \in V(G)$ with $\pi(x) \neq \pi(y)$ we have
    \[
        \alambda_G(x, y) \geq \alambda_{H / {\sim}}(\pi(x), \pi(y)).
    \]
\end{lemma}

\begin{proof}
    We need to orient all edges in $||K||$.
    For all $xy \in K$ that satisfy $|E_G(x, y)| = \infty$, we let $A \cupdot B = E_G(x, y)$ be a partition into two infinite sets.
    Then we orient the edges of $A$ from $x$ to $y$ and the edges of $B$ from $y$ to $x$.
    All other edges of $||K||$ can be oriented arbitrarily.
    Then \cref{eq:sim} implies
    \begin{equation}\label{eq:sim-inf}
        \forall xy \in \pairs{V(G)} \colon x \sim y \then \alambda_G(x, y) = \infty,
    \end{equation}
    as $\vv{H} \subseteq \vv{G}$ and $E_G(x, y) \neq \emptyset$ implies $xy \in K$ for all $xy \in \pairs{V(G)}$ with $x \sim y$, since $\sim$ is contractible in $H$.

    Let $x$ and $y$ be two distinct vertices of $G$ with $\pi(x) \neq \pi(y)$.
    Let $l \in \N$ with $l \leq \alambda_{H / {\sim}}(\pi(x), \pi(y))$.
    It suffices to show that $\alambda_G(x, y) \geq l$.
    Let $\PP' \coloneq \range{P'_1}{P'_l}$ be a set of pairwise edge-disjoint directed paths from $\pi(x)$ to $\pi(y)$ in $\vv{H / {\sim}}$.
    We will recursively define directed paths $P_1, \dots, P_l$ from $x$ to $y$ in $\vv{G}$, such that $P_i$ is edge-disjoint from all directed paths in $\singleset{P_1, \dots, P_{i - 1}} \cup (\PP' \setminus \singleset{P'_i})$ for $i = 1, \dots, l$.

    Assume that the directed paths $P_1, \dots, P_n$ are already defined for some $0 \leq n < l$.
    Consider the directed path $P'_{n + 1}$ from $\pi(x)$ to $\pi(y)$ in $H / {\sim}$, which is edge-disjoint from all directed paths in $\PP \coloneq \singleset{P_1, \dots, P_n} \cup (\PP' \setminus \singleset{P'_{n + 1}})$.

    We want to build a directed path $P_{n + 1}$ from $x$ to $y$ in $\vv{G}$, which is still edge-disjoint from all directed paths in $\PP$.
    Let $e_1, \dots, e_m \in E(P'_{n + 1})$ be the edges along $P'_{n + 1}$ enumerated in order from $\pi(x)$ to $\pi(y)$.
    Our plan is now to find directed paths in $\vv{G}$ that connect the vertex $x$ to the edge $e_1$ and $e_1$ to $e_2$ and so on, until connecting $e_m$ to $y$.

    For $i = 1, \dots, m$ we let $a_i, b_i \in V(G)$ be the vertices corresponding to the edge $e_i$ when considered in $\vv{G}$, i.e.\ $(a_i, b_i) = \vv{f}(e_i)$, were $\vv{f}$ is the incidence map of $\vv{G}$.
    We also define $b_0 \coloneq x$ and $a_{n + 1} \coloneq y$.
    By the choice of the enumeration of $e_1, \dots, e_m$, we have $b_i \sim a_{i + 1}$ for all $i \in [m]$.

    For all $i \in [m]$ we define a directed path $Q_i$ from $b_i$ to $a_{i + 1}$ in $\vv{G}$ as follows.
    If we have $b_i = a_{i + 1}$, then we simply let $Q_i$ be the trivial directed path containing just $b_i$.
    If on the other hand $b_i \neq a_{i + 1}$, then $b_i a_{i + 1} \in \pairs{V(G)}$ and since $|E(\bigcup \PP)|$ is finite, we can use \cref{eq:sim-inf}, to find a directed path $Q_i$ from $b_i$ to $a_{i + 1}$ in $\vv{G}$, which is edge-disjoint from all directed paths in $\PP$.

    In the end we consider $Q_0 \cup \singleset{e_1} \cup Q_1 \cup \dots \cup \singleset{e_m} \cup Q_m \subseteq \vv{G}$, which by construction contains a directed path $P_{n + 1}$ from $x$ to $y$ in $\vv{G}$.
    Then $P_{n + 1}$ is edge-disjoint from all directed paths in $\PP$, since all $Q_i$ for $i \in [m]$ are and the edges $e_i$ belong to $P'_{n + 1}$ for $i = 1, \dots, m$, which is edge-disjoint from all directed paths in $\PP$ as well.
\end{proof}

\begin{corollary}\label{cor:edge-decontraction}
    Let $G$ be a countable graph and $\sim$ be contractible in $G$.
    Assume there exists a well-balanced orientation $\vv{G / {\sim}}$.
    Assume further that
    \[
        \forall xy \in \pairs{V(G)} \colon x \sim y \then \alambda_G(x, y) = \infty.
    \]
    Then $\vv{G}$ is well-balanced.
\end{corollary}

\begin{proof}
  Let $x$ and $y$ be two distinct vertices of $G$.
    We want to show that $\vv{G}$ is well-balanced between $x$ and $y$.
    If $\pi(x) = \pi(y)$, then $x \sim y$, so we are done by the assumption.

    So consider the case $\pi(x) \neq \pi(y)$. Applying \cref{thm:1} with $K = \emptyset$, we get for all distinct vertices $x, y \in V(G)$ with $\pi(x) \neq \pi(y)$
    \begin{equation}\label{eq:thm-1-applied}
        \lambda_{G / {\sim}}(\pi(x), \pi(y)) \geq \lambda_G(x, y).
    \end{equation}
Applying \cref{thm:2} again with $K = \emptyset$ and the orientation on $G$ induced by $\vv{G / {\sim}}$, we have for all distinct vertices $x, y \in V(G)$ with $\pi(x) \neq \pi(y)$
    \begin{equation}\label{eq:thm-2-applied}
        \alambda_G(x, y) \geq \alambda_{G / {\sim}}(\pi(x), \pi(y)).
    \end{equation}

Thus, if $\lambda_G(x, y) = \infty$, the \cref{eq:thm-1-applied} and \cref{eq:thm-2-applied} imply $\alambda_G(x, y) = \infty$, too. And if $\lambda_G(x, y) < \infty$, then $\alambda_G(x, y) < \infty$, so $\alambda_{G / {\sim}}(\pi(x), \pi(y)) < \infty$ by \cref{eq:thm-2-applied}, and therefore
    \begin{align*}
        \alambda_G(x, y)
        &\geq \alambda_{G / {\sim}}(\pi(x), \pi(y))\\
        &\geq \floor{\frac{\lambda_{G / {\sim}}(\pi(x), \pi(y))}{2}}\\
        &\geq \floor{\frac{\lambda_G(x, y)}{2}}
    \end{align*}
    where the last inequality follows from \cref{eq:thm-1-applied}.
\end{proof}

\begin{corollary}\label{cor:edge-deletion}
    Let $G$ be a countable graph, $K$ be a skeleton in $G$ and $H \coloneq G - ||K||$ be such that
    \[
        \forall xy \in K \colon \lambda_H(x, y) = \infty.
    \]
    If there exists a well-balanced orientation $\vv{H}$,
    then there exists a well-balanced orientation $\vv{G}$ such that $\vv{G} - ||K|| = \vv{H}$.
\end{corollary}

\begin{proof}
    We apply \cref{thm:1} for the trivial\footnote{
        In the sense that $x \sim y \iff x = y$ for all $x, y \in V(G)$.
    } equivalence relation $\sim$ (which is always contractible) and get for all distinct vertices $x, y \in V(G)$
    \[
        \lambda_{H / {\sim}}(\dot{x}, \dot{y}) \geq \lambda_G(x, y).
    \]
    By \cref{thm:2} applied to the same equivalence relation $\sim$ there exists an orientation $\vv{G}$ such that $\vv{G} - ||K|| = \vv{H}$ and for all distinct vertices $x, y \in V(G)$
    \[
        \alambda_G(x, y) \geq \alambda_{H / {\sim}}(\dot{x}, \dot{y}).
    \]
    As $H / {\sim}$ is isomorphic to $H$ and $\vv{H}$ is well-balanced, the proof continues along the lines of the previous corollary.
\end{proof}

\begin{corollary}\label{cor:edge-deletion-decontraction}
    Let $G$ be a countable graph and $v, w \in V(G)$ be two distinct vertices such that
    \[
        |E_G(v, w)| = \infty.
    \]
    Consider $H \coloneq G - E_G(v, w)$. If there exists a well-balanced orientation $\vv{H / \singleset{v, w}}$, then there exists a well-balanced orientation $\vv{G}$ such that $\vv{G} - E_G(v, w) = \vv{H}$.
\end{corollary}

\begin{proof}
    We let $K \coloneq \singleset{vw}$ be a skeleton in $G$.
    Then $||K|| = E_G(v, w)$.
    We also define $\sim$ to be the equivalence relation on $V(H)$ generated by $\singleset{v, w}$.
    Then we have
    \[
        \forall xy \in \pairs{V(G)} \colon x \sim y \then |E_G(x, y)| = \infty.
    \]
    These properties again allow us to apply \cref{thm:1} and \cref{thm:2} similar to the previous two proofs.
\end{proof}

The following lemma forms the induction start in the proof of \cref{thm:rayless-intro}.

\begin{lemma}\label{lem:induction-start}
    Any countable graph with finitely many vertices only (but potentially infinitely many edges) has a well-balanced orientation.
\end{lemma}

\begin{proof}
    We prove the lemma by induction on $|V(G)| \in \N$.
    For $|V(G)| \leq 1$ the claim is trivial.
    In the induction step we have $|V(G)| \geq 2$.

    If there exists $xy \in \pairs{V(G)}$ such that $E_G(x, y)$ is infinite, we apply the induction hypothesis to $(G - E_G(x, y)) / \xy$.
    By \cref{cor:edge-deletion-decontraction} we are done.

    Otherwise the sets $E_G(x, y)$ are finite for all $xy \in \pairs{V(G)}$.
    But then $G$ is finite, so the result follows from Nash-Williams' strong orientation theorem \cite{NashWilliams1960}.
\end{proof}

\subsection{On the simple graph \texorpdfstring{$L(\protect\vv{G}, N)$}{L(G, N)}}

In this section we investigate the following simple graph.

\begin{definition}
    Let $\vv{G}$ be an oriented graph and let $N \subseteq V(G)$.
    We define a simple graph $L(\vv{G}, N)$ as follows
    \begin{align*}
        V(L(\vv{G}, N)) &\coloneq N\\
        E(L(\vv{G}, N)) &\coloneq \set{xy \in \pairs{N}}{\alambda_G(x, y) \geq 1 \text{ or } \alambda_G(y, x) \geq 1}.
    \end{align*}
\end{definition}

The following lemma considers a special case of this notion.

\begin{lemma}\label{lem:alambda-or}
    Let $G$ be a connected graph.
    Let $x, y \in V(G)$ be two distinct and independent vertices.
    Assume there exists a well-balanced orientation $\vv{G / \xy}$.
    Then
    \[
        \alambda_G(x, y) \geq 1 \qquad \text{or} \qquad \alambda_G(y, x) \geq 1.
    \]
\end{lemma}

\begin{proof}
    For the proof we will write $G^* \coloneq G / \xy$ and $*$ for the vertex of $G^*$ corresponding to the equivalence class $\xy$.
    We also define
    \[
        X \coloneq \set{v \in V(G - x)}{\alambda_G(x, v) \geq 1 \text{ and } \alambda_G(v, x) \geq 1}.
    \]

    Since $G$ is connected, there exists a path $P$ from $x$ to $y$ in $G$.
    Let $x'$ be the last vertex along $P$ that is in $X \cup \singleset{x}$ ($x$ is a candidate).
    If $y \in X$ we are done, so we may assume that $y \notin X$.
    Now $y \notin X \cup \singleset{x}$ implies $x' \neq y$, so we can choose $u$ to be the first vertex along $P$ after $x'$ (in particular $u \notin X \cup \singleset{x}$).

    Let $e$ be the edge of $P$ between $x'$ and $u$.
    We may assume that $\vv{f}(e) = (x', u)$, as flipping the orientation\footnote{
        I.e.\ reversing the orientation of each edge.
    } $\vv{G / \xy}$ will still be well-balanced and the definition of $X$ and the claim of the lemma are symmetric.
    In particular $\alambda_G(x', u) \geq 1$.

    \begin{claim}
        We have $\alambda_G(x, u) \geq 1$.
    \end{claim}

    \begin{proof}[Proof of claim]\renewcommand{\qedsymbol}{$\Diamond$}
        If $x' = x$, then $\alambda_G(x, u) = \alambda_G(x', u) \geq 1$.
        In the case $x' \in X$ we argue similarly
        \[
            \alambda_G(x, u) \geq \minset{\alambda_G(x, x'), \alambda_G(x', u)} \geq \minset{1, 1} = 1,
        \]
        where $\alambda_G(x, x') \geq 1$ comes from the definition of $X$.
    \end{proof}

    If $u = y$, then we are immediately done using the claim.
    Therefore, we may assume that $u \neq y$.
    Now we have both $u \neq x$ (since $u \notin X \cup \singleset{x}$) and $u \neq y$, which implies $\dot{u} \neq *$ in $G^*$.
    Hence, the cycle in $G^*$ that is induced by $P$ (and therefore in particular contains both $\dot{u}$ and $*$), shows $\lambda_{G^*}(\dot{u}, *) \geq 2$.
    Since the orientation $\vv{G^*}$ is well-balanced, we get $\alambda_{G^*}(\dot{u}, *) \geq 1$.
    Let $Q$ be a directed path from $\dot{u}$ to $*$ in $\vv{G^*}$.

    \begin{claim}
        The lift of $Q$ is a directed path from $u$ to $y$ in $\vv{G}$.
    \end{claim}

    \begin{proof}[Proof of claim]\renewcommand{\qedsymbol}{$\Diamond$}
        First note that the lift of $Q$ will be a directed path (from $u$ to $y$ or from $u$ to $x$), as there is only one edge of $Q$ adjacent to $* = \xy$, which is the only non trivial equivalence class.
        Suppose for a contradiction that the lift of $Q$ is a directed path from $u$ to $x$ in $\vv{G}$.
        Then the existence of this lift shows $\alambda_G(u, x) \geq 1$ and since $\alambda_G(x, u) \geq 1$ as well (by the first claim), we have $u \in X$.
        But this contradicts $u \notin X \cup \singleset{x}$.
    \end{proof}

    The second claim shows $\alambda_G(u, y) \geq 1$.
    Together with the first claim we are done.
\end{proof}

\begin{theorem}\label{thm:L-connected}
    Let $G$ be a connected graph and let $N$ be contractible and non-empty in $G$.
    Assume there exists a well-balanced orientation $\vv{G / N}$.
    Then $L(\vv{G}, N)$ is connected.
\end{theorem}

\begin{proof}
    The simple graph $L(\vv{G}, N)$ is non-empty, since $N \neq \emptyset$, so we can let $X \subseteq N$ be the vertices of a component of $L(\vv{G}, N)$.
    Suppose for a contradiction that its complement $Y \coloneq N \setminus X$ is non-empty.

    Consider the equivalence relation $\sim$ on $V(G)$ generated by $X$ and $Y$ and the graph $H \coloneq G / {\sim}$, which is connected by \cref{lem:connected}, since $G$ is connected.
    As $X$ and $Y$ are equivalence classes under $\sim$, we have $X, Y \in V(H)$.
    Note that $H / \singleset{X, Y}$ is (isomorphic to) $G / N$, for which there exists a well-balanced orientation by assumption.
    Therefore, we can apply \cref{lem:alambda-or} (to $H$ and $X, Y \in V(H)$), to see that $\alambda_{H}(X, Y) \geq 1$ or $\alambda_{H}(Y, X) \geq 1$.
    We may assume that $\alambda_{H}(X, Y) \geq 1$.
    Let $P$ be a directed path from $X$ to $Y$ in $\vv{H}$.
    Then the lift of $P$ gives a directed path from a vertex in $X$ to a vertex in $Y$ in $\vv{G}$.
    This contradicts the maximality of $X$.
\end{proof}

\section{Well-balanced orientations for countable rayless graphs }
\label{section:main-theorem}

Finally we turn our attention to rayless graphs. Our main result in this section is that every countable rayless graph has a well-balanced orientation.

\subsection{The order of a rayless graph}
\label{section:rayless}

We follow \cite{Halin1998} to recall the definitions and properties related to rayless graphs.
Note that although Halin restricts to simple graphs, everything can be stated for graphs with parallel edges without problems.

\begin{definition}[{Schmidt 1982, \cite{Schmidt1982EinRF}, \cite[Definition~3.1]{Halin1998}}]
    For all ordinals $\alpha$, let us recursively define a class $A(\alpha)$ of graphs as follows:
    \begin{enumerate}
        \item
            The class $A(0)$ consists of all graphs with finitely many vertices (the set of edges may be infinite).
        \item
            If $\alpha > 0$ and $A(\beta)$ has been defined for every $\beta < \alpha$, then a graph $G$ is defined to be in $A(\alpha)$ if and only if there exists a finite $U \subseteq V(G)$ such that every component of $G - U$ is in some $A(\beta)$ with $\beta < \alpha$.
    \end{enumerate}

    Let $A$ denote the union of the classes $A(\alpha)$.
    Then for every $G \in A$ there exists a smallest $\alpha$ with $G \in A(\alpha)$.
    This $\alpha$ is called the \textbf{order} of $G$ and denoted by $o(G)$.

    If $o(G) > 0$ there is a finite $U \subseteq V(G)$ such that $o(C) < o(G)$ for every $C \in \C(G - U)$.
    Every such $U$ is called \textbf{reducing} in $G$.
\end{definition}

\begin{theorem}[{Schmidt 1982, \cite{Schmidt1982EinRF}}]
    The class $A$ is the class of all rayless graphs.
\end{theorem}

\begin{proof}
    If $G$ contains a ray, then $G$ must have infinitely many vertices, so $G \notin A(0)$ and therefore $o(G) > 0$.
    The remaining proof is exactly the same as Halin's proof of \cite[Proposition~3.2]{Halin1998} and multiple edges make no difference.
\end{proof}

The proofs of the following statements are also exactly the same and therefore will not be repeated here.

\begin{theorem}[{Schmidt 1982, \cite{Schmidt1982EinRF}, \cite[Proposition~3.3]{Halin1998}}]\label{thm:halin-subset}
    If $G \in A$ and $H \subseteq G$, then $H \in A$ and $o(H) \leq o(G)$.
\end{theorem}

\begin{theorem}[{Schmidt 1982, \cite{Schmidt1982EinRF}, \cite[Proposition~3.5]{Halin1998}}]
    For a rayless graph $G$ with infinite vertex set and an ordinal $\alpha$ the following statements are equivalent.
    \begin{itemize}
        \item
            We have $o(G) > \alpha$.
        \item
            If $U$ is reducing in $G$, then $G - U$ has infinitely many components of order $\geq \alpha$.
        \item
            There are infinitely many disjoint connected subgraphs of $G$ with order $\alpha$.
    \end{itemize}
\end{theorem}

\begin{corollary}[{\cite[Corollary~3.8]{Halin1998}}]\label{cor:halin-order-cases}
    Let $G$ be rayless and $U \subsetneq V(G)$ be finite.
    Further let $\family{G_i}{i \in I}$ be the distinct components of $G - U$.
    Then we have
    \[
        o(G) = \begin{cases}
            \max_{i \in I} o(G_i) &\parbox[t]{.6\textwidth}{if this maximum exists and is attained\\only for finitely many $i \in I$,}\\
            \sup_{i \in I} o(G_i) &\parbox[t]{.6\textwidth}{if the $o(G_i)$ have no maximum,}\\
            \max_{i \in I} o(G_i) + 1 &\parbox[t]{.6\textwidth}{if this maximum exists and is attained by infinitely many $i \in I$.}
        \end{cases}
    \]
\end{corollary}

The following lemmas are not taken from Halin's paper and therefore, we will give proofs.

\begin{lemma}[{Schmidt 1982, \cite{Schmidt1982EinRF}}]\label{lem:rayless-finite-addition}
    Let $G$ be a graph and $L \subseteq V(G)$ be finite.
    If $G - L$ is rayless, then $G$ is rayless with $o(G) = o(G - L)$.
\end{lemma}

\begin{remark}
    This lemma can be viewed as being inverse to \cite[Corollary~3.7]{Halin1998}, which states that if $G$ is rayless and $L \subseteq V(G)$ is finite, then $o(G - L) = o(G)$.
\end{remark}

\begin{proof}[Proof of \cref{lem:rayless-finite-addition}]
    We will prove that $G$ is rayless with $o(G) \leq o(G - L)$.
    The claimed equality follows from \cref{thm:halin-subset}.

    Let $U$ be reducing in $G - L$ and $C \in \C(G - (L \cup U))$.
    We will show that $C$ is rayless with $o(C) < o(G - L)$, which shows the claim, since $L \cup U$ is finite by assumption.
    We have $G - (L \cup U) = (G - L) - U$, so $C \in \C((G - L) - U)$.
    Therefore, we have $o(C) < o(G - L)$, since $U$ is reducing in $G - L$.
\end{proof}

\begin{lemma}\label{lem:rayless-contraction}
    Let $G$ be a rayless graph and $\sim$ be contractible in $G$, such that there are only finitely many non-trivial
  equivalence classes.
    Then $G / {\sim}$ is also rayless with $o(G / {\sim}) \leq o(G)$.
\end{lemma}

\begin{proof}
    We write $H \coloneq G / {\sim}$ and $\NN$ for the set of non trivial equivalence classes under ${\sim}$ and note that $\NN \subseteq V(H)$.
    Thus, the vertices of $H - \NN$ consist only of trivial equivalence classes $\dot{x}$, which we identify with their unique vertex $x \in V(G)$.
    This allows us to consider $H - \NN \subseteq G$.
    By \cref{thm:halin-subset} $H - \NN$ is rayless with $o(H - \NN) \leq o(G)$.
    Since by assumption the set $\NN$ is finite, we can apply \cref{lem:rayless-finite-addition} to see that $H$ is rayless as well with $o(H) = o(H - \NN) \leq o(G)$.
\end{proof}

\subsection{Constructing well-balanced orientations}

In this section we will prove the core result about countable, rayless graphs, which, via the Reduction Theorem~\ref{thm:reduction-to-countable-case} to be proved in \cref{section:reduction} below implies the main result of this paper.

\begin{theorem}
\label{thm_countablerayless}
      Every countable, rayless graph admits a well-balanced orientation.
\end{theorem}

\begin{proof}
    Let $G$ be a countable rayless graph. We may assume that $G$ is connected.
    We will prove the assertion by transfinite induction on $o(G)$.
    The base case $o(G) = 0$, i.e.\ when $G$ has finitely many vertices, is covered by \cref{lem:induction-start}.
    If $o(G) > 0$, fix a finite reducing set $U$ in $G$.
    Let $f$ denote the incidence map of $G$.
    For all $C \in \C(G - U)$ we define a subgraph $\overline{C} \subseteq G$ as follows.
    \begin{align*}
        V(\overline{C}) &\coloneq V(C) \cup N(C)\\
        E(\overline{C}) &\coloneq E(C) \cup E(C, U)\\
        f_{\overline{C}} &\coloneq f|_{E(\overline{C})}.
    \end{align*}
    Then $N(C) \subseteq U$ is an independent set of vertices in $\overline{C}$.
    For every $xy \in \pairs{U}$ we define
    \[
        \C(xy) \coloneq \set{C \in \C(G - U)}{\xy \subseteq N(C)}.
    \]
    Define a simple graph $R$ by
    \begin{align*}
        V(R) &\coloneq U\\
        E(R) &\coloneq \set{xy \in \pairs{U}}{|\C(xy)| = \infty}
    \end{align*}
    and let $\sim$ denote the equivalence relation on $V(G)$ induced by the components of $R$ (and singletons of $V(G) - U$).
    Then for all $xy \in \pairs{U}$ with $x \sim y$ there exist
    $x_0, \dots, x_n \in U$ such that $x_0 = x$, $x_n = y$ and
    \begin{equation}\label{eq:N-generated-by-E-R}
        \forall i \in [n - 1] \colon |\C(x_i x_{i + 1})| = \infty .
    \end{equation}
   
    Without loss of generality, we may assume that $\sim$ is contractible in $G$: Indeed, $E(R)$ is a skeleton in $G$ and $\sim$ is contractible in $G - ||E(R)||$. Then \cref{eq:N-generated-by-E-R} yields $\lambda_{G - ||E(R)||}(x, y) = \infty$ for all $x \sim y$ (with transitivity through the components), so if $G - ||E(R)||$ has a well-balanced orientation, then so does $G$ by \cref{cor:edge-deletion}.

    Define
    \[
        \Cinf \coloneq \set{C \in \C}{\forall x, y \in N(C) \colon x \sim y},
    \]
    and let $\Cless$ be its complement in $\C$.
   
    \begin{claim}
        The set $\Cless$ is finite.
    \end{claim}

    \begin{proof}[Proof of claim]\renewcommand{\qedsymbol}{$\Diamond$}
        We rewrite $\Cless$ as follows.
        \begin{align*}
            \Cless
            &= \set{C \in \C}{\exists x, y \in N(C) \colon x \not\sim y} \\
            &\subseteq \set{C \in \C}{\exists x, y \in N(C) \colon x \neq y \text{ and } xy \notin E(R)} \\
            &= \set{C \in \C}{\exists xy \in \pairs{U} \colon \singleset{x, y} \subseteq N(C) \text{ and } xy \notin E(R)} \\
            &= \bigcup_{xy \in \pairs{U} \setminus E(R)} \C(xy).
        \end{align*}
        This final set is finite, as $|\C(xy)| < \infty$ for all $xy \notin E(R)$ by definition.
    \end{proof}
    
    Consider the induced subgraph
    \[
        H \coloneq G \left[ U \cup V \left( \bigcup \Cless \right) \right].
    \]
    With this definition we have $\Cless = \C(H - U)$.
    Now $\Cless$ is a finite collection of components of $G-U$, and hence $o(H) < o(G)$ by the first case of \cref{cor:halin-order-cases}.

    We define $\sim_H$ as the restriction of $\sim$ to the vertices of $H$ and note that the graph $H / {\sim_H}$ is a subgraph of $G / {\sim}$.
    The definition of $\Cinf$ allows us to also consider the graphs $\overline{C} / N(C)$ for $C \in \Cinf$ as subgraphs of $G / {\sim}$.\footnote{
        If we consider $\overline{C} / N(C)$ as a subgraph of $G / {\sim}$, then $N(C) \notin \overline{C} / N(C)$ in general, because the corresponding equivalence class in $V(G)$ might contain more vertices than $N(C)$.
    }

    The plan for the remaining part of the proof is now as follows.
    \begin{enumerate}
        \itemlabel{it:step-1}
            We will find well-balanced orientations $\vv{H / {\sim_H}}$ and $\vv{\overline{C} / N(C)}$ for all $C \in \Cinf$ in such a way, that
            \begin{equation}\label{eq:step-1}
                \forall xy \in \pairs{U} \colon x \sim y \then \alambda_G(x, y) = \infty
            \end{equation}
            is satisfied.
         \itemlabel{it:step-2}
            We conclude that their union induces a well-balanced orientation $\vv{G / {\sim}}$.
         \itemlabel{it:step-3}
            Finally the additional property \cref{eq:step-1} allows us to apply \cref{cor:edge-decontraction}, so that we can lift the well-balanced orientation from $\vv{G / {\sim}}$ to $\vv{G}$.
    \end{enumerate}

    We first deal with \ref{it:step-1}.
    Since $o(H / {\sim_H}) \leq o(H)$ by \cref{lem:rayless-contraction} and since $o(H) < o(G)$ as observed above, there exists a well-balanced orientation $\vv{H / {\sim_H}}$ by the induction hypothesis.
    Similarly, since $\overline{C} / N(C)$ is rayless with $o(\overline{C} / N(C)) = o(C)< o(G)$ by \cref{lem:rayless-finite-addition}, we get a well-balanced orientation  $\vv{\overline{C} / N(C)}$ by the induction hypothesis. The main challenge is to ensure \cref{eq:step-1}. For this, we  flip some of those orientations for some $C \in \Cinf$ determined as follows:    
    For each $xy \in \pairs{U}$ we define
    \[
        \D(xy) \coloneq \set{C \in \C(xy) \setminus \Cless}{\alambda_{\overline{C}}(x, y) \geq 1 \text{ or } \alambda_{\overline{C}}(y, x) \geq 1}.
    \]
    Define a simple graph $L$ by
    \begin{align*}
        V(L) &\coloneq U\\
        E(L) &\coloneq \set{xy \in \pairs{U}}{|\D(xy)| = \infty}.
    \end{align*}
    The number of edges of $L$ is finite and for each edge $xy \in E(L)$ the set $\D(xy)$ is infinite by definition.
Pick infinite subsets $\D'(xy) \subseteq \D(xy)$ such that the family $\{\D'(xy) \mid xy \in E(L)\}$ is pairwise disjoint. 
    Then for a fixed $xy \in E(L)$ we can flip the orientations of components in $\D'(xy)$, to assume that there are infinitely many components $C \in \D'(xy)$ with $\alambda_{\overline{C}}(x, y) \geq 1$ and also infinitely many components $C \in \D'(xy)$ with $\alambda_{\overline{C}}(y, x) \geq 1$.

    Since the sets $\D'(xy)$ for $xy \in E(L)$ are pairwise disjoint, we can assume this for all $xy \in E(L)$ simultaneously, i.e.\ in the (together with $\vv{H}$) induced orientation $\vv{G}$ we have
    \begin{equation}\label{eq:alambda-E-L}
        \forall xy \in E(L) \colon \alambda_{G}(x, y) = \infty \text{ and } \alambda_{G}(y, x) = \infty.
    \end{equation}
    We now show that this already suffices to prove \cref{eq:step-1} (which was stated for $\sim$ and not just $E(L)$).
    Let $xy \in \pairs{U}$ with $x \sim y$.
    We need to show that $\alambda_{G}(x, y) = \infty$.
    We may assume that $xy \in E(R)$, by induction on the vertices given by \cref{eq:N-generated-by-E-R}.
    Hence, $\C(xy)$ is infinite.
    Now if $xy \in E(L)$, we are done by \cref{eq:alambda-E-L}.
    But in general we might have $E(L) \subsetneq E(R)$, and in this case we need the following claim:

    \begin{claim}
        The vertices $x$ and $y$ are contained in the same component of $L$.
    \end{claim}

    \begin{proof}[Proof of claim]\renewcommand{\qedsymbol}{$\Diamond$}
        Denote by $KU$ the complete simple graph on the vertices $U$.
        For all $C \in \C(xy) \setminus \Cless$ the simple graph $L(\vv{\overline{C}}, N(C))$ is a simple subgraph of $KU$ (since $N(C) \subseteq U$).
        But there are only finitely many simple subgraphs of $KU$ and $\C(xy) \setminus \Cless$ is infinite (as we assumed $\C(xy)$ to be infinite above).
        So by the pigeonhole principle, there is an $C \in \C(xy) \setminus \Cless$ such that its corresponding subgraph $L(\vv{\overline{C}}, N(C)) \subseteq KU$ is attained by infinitely many other components in $\C(xy) \setminus \Cless$ as well.

        For all edges $x'y'$ of $L(\vv{\overline{C}}, N(C))$ we have by definition $\alambda_{\overline{C}}(x', y') \geq 1$ or $\alambda_{\overline{C}}(y', x') \geq 1$.
        Therefore, $\D(x'y')$ is infinite since $L(\vv{\overline{C}}, N(C))$ (and in particular the edge $x'y'$) is attained infinitely often.
        Hence, $x'y' \in E(L)$ by definition of $E(L)$.
        This proves
        \[
            E(L(\vv{\overline{C}}, N(C))) \subseteq E(L).
        \]
        Now $x, y \in V(L(\vv{\overline{C}}, N(C)))$ as $\xy \subseteq N(C)$, together with $L(\vv{\overline{C}}, N(C))$ being connected by \cref{thm:L-connected}, imply the claim.
    \end{proof}

    Using the claim, we can now let $P$ be a path from $x$ to $y$ in $L$.
    Then for each edge of $P$, we apply \cref{eq:alambda-E-L} and conclude $\alambda_G(x, y) = \infty$ by transitivity.
    This concludes the argument for step \ref{it:step-1}.

    We now turn towards  \ref{it:step-2}.  The graph $H / {\sim_H}$ together with all $\overline{C} / N(C)$ for $C \in \Cinf$ induces a tree-decomposition of $G$ of adhesion $\leq 1$ as follows.
    We define a tree $T$ with root $r$ and for every component $C \in \Cinf$ a child $t_C$ of the root.
    The resulting tree is a star with center $r$.
    We also let $V_r \coloneq V(H / {\sim_H})$ and $V_{t_C} \coloneq V(\overline{C} / N(C))$ (regarded as a subset of $V(G / {\sim})$) for all $C \in \Cinf$.

    \begin{claim}
        $(T, \family{V_t}{t \in V(T)})$ is a tree-decomposition of adhesion $\leq 1$ of $G / {\sim}$.
    \end{claim}

    \begin{proof}[Proof of claim]\renewcommand{\qedsymbol}{$\Diamond$}
        Condition \ref{it:tree-decomposition-t1} is satisfied, as we have $\bigcup_{t \in V(T)} V_t = V(G / {\sim}$.
        Condition \ref{it:tree-decomposition-t2} is satisfied as well and condition \ref{it:tree-decomposition-t3} is immediate from the fact that for any two $C,C' \in \Cinf$ we have $\overline{C} \cap \overline{C'} \subseteq U \subseteq H$.

        Now, we verify that the adhesion is $\leq 1$.
        As $T$ is a star with center $r$, the adhesion sets are given by the intersections $V_r \cap V_{t_C} = V(H / {\sim_H}) \cap V(\overline{C} / N(C))$ for $C \in \Cinf$. 
        Let $C \in \Cinf$.
        By definition of $\Cinf$ there exists an equivalence class $N \subseteq U$ with $N(C) \subseteq N$.
        As $G$ is connected, $N(C) \neq \emptyset$ and since we regard $\overline{C} / N(C) \subseteq G / {\sim}$ we have $N \in V(\overline{C} / N(C))$.
        Hence, $V(H / {\sim_H}) \cap V(\overline{C} / N(C)) = \singleset{N}$ contains (at most) one element.
    \end{proof}

    Hence, we obtain a well-balanced orientation $\vv{G / {\sim}}$ by \cref{lem:well-balanced-cut-decomposition}.
    Finally, for step \ref{it:step-3}, it suffices to observe that due to \ref{it:step-2} and \cref{eq:step-1}, the requirements of \cref{cor:edge-decontraction} are met.
    Thus, $G$ has a well-balanced orientation, and the induction step is complete.
\end{proof}

\section{Reduction to the countable case}\label{section:reduction}

In this section we establish \cref{thm:reduction-to-countable-case} announced in the introduction, that it suffices to prove the strong orientation conjecture for countable graphs.
Recall the following definitions from \cite[p.~262]{Laviolette2005}:

\begin{definition}\label{def:decomposition}
    Let $G$ be a graph and $\family{G_i}{i \in I}$ be a family of subgraphs.
    We call $\family{G_i}{i \in I}$ a \textbf{decomposition} of $G$, if the graphs are pairwise edge-disjoint and $G = \bigcup_{i \in I} G_i$.
    The graphs $G_i$ for $i \in I$ are called the \textbf{fragments} of the decomposition.
    The decomposition is called \textbf{bond-faithful} if any finite bond of $G_i$ is also a bond of $G$ for all $i \in I$.
\end{definition}

\begin{theorem}[{Laviolette's decomposition theorem, \cite[Theorem~3]{Thomassen2016}}]\label{thm:laviolette}
    Every graph has a bond-faithful decomposition in countable\footnote{
        Meaning countably many vertices and edges.
    } and connected subgraphs.\footnote{
        Thomassen's terminology differs: ``edge-decomposition'' translates to decomposition and ``cut-faithful'' translates to bond-faithful.
        Thomassen allows multiple parallel edges (as we do).
    }
\end{theorem}

We also need the following definitions for the next lemmas.

\begin{definition}
    Let $G$ be a graph.
    We write
    \[
        \C^E(G) \coloneq \set{C \in \C(G)}{E(C) \neq \emptyset}
    \]
    for the set of components of $G$ that contain at least one edge.
\end{definition}

\begin{definition}
    Let $G$ be a graph and $\family{G_i}{i \in I}$ be a decomposition of $G$.
    Let $x, y \in V(G)$ be two distinct vertices and $P$ be a path from $x$ to $y$ in $G$.
    We say that $P$ is \textbf{efficient (with respect to the decomposition $\family{G_i}{i \in I}$)} if
    \[
        \forall i \in I \colon |\C^E(P \cap G_i)| \leq 1.
    \]
\end{definition}

In other words, we require that $P$ meets each $G_i$ in (singleton vertices and) at most one non-trivial segment of $P$.
The notion of efficient paths is useful for the following lemma.

\begin{lemma}\label{lem:decomposition-connectivity}
    Let $G$ be a graph and let $\family{G_i}{i \in I}$ be a bond-faithful decomposition of $G$ into countable and connected subgraphs.
    Let $P$ be an efficient $x-y$-path in $G$.
    Furthermore let $x_0, \dots, x_{n + 1} \in V(P)$ be the pairwise distinct vertices enumerated along the order of $P$ such that $x_0 = x$, $x_{n + 1} = y$ and for all $m \in [n]$
    \[
        x_m P x_{m + 1} \in \C^E(P \cap G_{i_m})
    \]
    for some $i_m \in I$.
    Then for all $m \in [n]$
    \[
        \lambda_{G_{i_m}}(x_m, x_{m + 1}) \geq \minset{\lambda_G(x, y), \aleph_0}.
    \]
\end{lemma}

\begin{proof}
    Let $\ell \coloneq \minset{\lambda_G(x, y), \aleph_0}$ and suppose for a contradiction that there exists an $m \in [n]$ with $\lambda_{G_{i_m}}(x_m, x_{m + 1}) < \ell$.

    Now $\ell \leq \aleph_0$, so $\lambda_{G_{i_m}}(x_m, x_{m + 1}) < \ell \leq \aleph_0$ is finite.
    Hence, there exists a finite bond $B$ in $G_{i_m}$ with $|B| = \lambda_{G_{i_m}}(x_m, x_{m + 1})$, which separates $x_m$ from $x_{m + 1}$.
    Since $\family{G_i}{i \in I}$ is a bond-faithful decomposition of $G$, $B$ is also a bond in $G$, separating $x_m$ from $x_{m + 1}$.

    Let $X$ and $Y$ be the two components of $G - B$ such that $x_m \in V(X)$ and $x_{m + 1} \in V(Y)$.
    We have $x_m P x_{m + 1} \in \C^E(P \cap G_{i_m})$ and since $P$ is efficient, this implies that both parts $x P x_m$ and $x_{m + 1} P y$ lie in $G - E(G_{i_m}) \subseteq G - B$ and therefore, $x \in V(X)$ and $y \in V(Y)$ as well.
    But this implies that $B$ separates $x$ from $y$ in $G$ with only $|B| = \lambda_{G_{i_m}}(x_m, x_{m + 1}) < \ell \leq \lambda_G(x, y)$ many edges, which is a contradiction.
\end{proof}

The next lemma shows that every path can be made efficient.

\begin{lemma}\label{lem:efficient}
    Let $G$ be a graph and let $\family{G_i}{i \in I}$ be a decomposition of $G$  into connected subgraphs.
    Let $x, y \in V(G)$ be two distinct vertices and $P$ be a path from $x$ to $y$ in $G$.
    Then there exists an efficient path $P'$ from $x$ to $y$ in $G$ such that
    \[
        \forall i \in I \colon E(P \cap G_i) = \emptyset \then E(P' \cap G_i) = \emptyset.
    \]
\end{lemma}

\begin{proof}
    Let $P'$ be a path from $x$ to $y$ in $G$ such that
    \[
        \forall i \in I \colon |\C^E(P' \cap G_i)| \leq |\C^E(P \cap G_i)|
    \]
    and $\sum_{i \in I} |\C^E(P' \cap G_i)| \in \N$ is minimal.
    Such a path exists, since $P$ is a candidate.
    The first property implies that $P'$ satisfies the implication from the claim, since $E(H) = \emptyset \iff |\C^E(H)| = 0$ for all graphs $H$.

    We are left to show that $P'$ is efficient.
    We suppose for a contradiction that there exists an $i \in I$ such that $|\C^E(P' \cap G_i)| \geq 2$.
    Let $a \in V(P' \cap G_i)$ be minimal along the order of $P'$ and $b \in V(P' \cap G_i)$ be maximal.
    By assumption $G_i$ is connected, so there exists a path $R$ path from $a$ to $b$ in $G_i$.
    Then we let $Q \coloneq x P' a R b P' y$ be the combined path from $x$ to $y$.

    Now $Q$ satisfies $|\C^E(Q \cap G_i)| = |\singleset{R}| = 1 < 2 \leq |\C^E(P' \cap G_i)|$ and $\forall j \in I, j \neq i \colon |\C^E(Q \cap G_j)| \leq |\C^E(P' \cap G_j)|$.
    Therefore, $Q$ contradicts the minimality of $\sum_{i \in I} |\C^E(P' \cap G_i)|$.
\end{proof}

\reductionIntro*

\begin{proof}
    Let $G \in \A$.
    We may assume that $G$ is connected, since the components are subgraphs and hence also in $\A$ and a well-balanced orientation on all components induces a well-balanced orientation on $G$.

    By \cref{thm:laviolette}, there exists a bond-faithful decomposition $\family{G_i}{i \in I}$ of $G$ into countable and connected subgraphs.
    The graphs $G_i$ for $i \in I$ are subgraphs of $G$ and hence in $\A$.
    They are countable, so by assumption there exist well-balanced orientations $\vv{G_i}$ of $G_i$ for all $i \in I$.
    We show that $\vv{G} \coloneq \bigcup_{i \in I} \vv{G_i}$ is a well-balanced orientation.

    Let $x, y \in V(G)$ be two distinct vertices.
    We distinguish two cases based on the cardinality of $\lambda_G(x, y)$.

    \begin{case}[$\lambda_G(x, y) \leq \aleph_0$]
        Since we assumed $G$ to be connected, we let $P$ be a path from $x$ to $y$ in $G$.
        By applying \cref{lem:efficient} we may assume that $P$ is efficient.
        Then we apply \cref{lem:decomposition-connectivity}.
        We get $\lambda_{G_{i_m}}(x_m, x_{m + 1}) \geq \minset{\lambda_G(x, y), \aleph_0} = \lambda_G(x, y)$ for all $m \in [n]$ with the notation as in \cref{lem:decomposition-connectivity}.

        Then we have for all $m \in [n]$ (using that $\vv{G_{i_m}}$ is well-balanced)
        \[
            \alambda_{G_{i_m}}(x_m, x_{m + 1}) \geq \begin{cases}
                \floor{\frac{\lambda_{G_{i_m}}(x_m, x_{m + 1})}{2}} \geq \floor{\frac{\lambda_G(x, y)}{2}} & \text{if } \lambda_{G_{i_m}}(x_m, x_{m + 1}) < \infty, \\
                \lambda_{G_{i_m}}(x_m, x_{m + 1}) \geq \lambda_G(x, y) & \text{else.}
            \end{cases}
        \]
        The second case can arise even if $\lambda_G(x, y) < \infty$, but in that case we have $\lambda_G(x, y) \geq \floor{\frac{\lambda_G(x, y)}{2}}$.
        Therefore, we get by transitivity
        \[
            \alambda_G(x, y) \geq \begin{cases}
                \floor{\frac{\lambda_G(x, y)}{2}} & \text{if } \lambda_G(x, y) < \infty \\
                \lambda_G(x, y) & \text{else.}
            \end{cases}
        \]
    \end{case}

    \begin{case}[$\lambda \coloneq \lambda_G(x, y) \geq \aleph_1$]
        For subsets $I' \subseteq I$ we write $G(I')$ for the subgraph $\bigcup_{i \in I'} G_i \subseteq G$.
        By transfinite recursion, we define for all ordinals $\alpha < \lambda$ pairwise disjoint subsets $(I_\alpha \colon \alpha < \lambda)$ of $I$ such that for all $\alpha < \lambda$
        \begin{enumerate}
            \item
                $I_\alpha$ is finite,
            \item
                $x, y \in V(G(I_\alpha))$ and
            \itemlabel{it:alambda-aleph_0}
                $\alambda_{G(I_\alpha)}(x, y) = \aleph_0$.
        \end{enumerate}
        Let $\alpha < \lambda$ and assume that all $(I_\beta \colon \beta < \alpha)$ have been defined.
        Let $J \coloneq \bigcup_{\beta < \alpha} I_\beta$ and $F \coloneq E(G(J))$.\footnote{
            For $\alpha = 0$ both sets are simply empty.
        }
        Since the graphs of the decomposition have only countably many edges and $\aleph_1 \leq \lambda$, we have
        \[
            |F| = \big|\bigcup_{\beta < \alpha} \bigcup_{k \in I_\beta} E(G_k)\big| \leq \alpha \cdot \aleph_0 \cdot \aleph_0 < \lambda = \lambda_G(x, y).
        \]
        Therefore there exists a path $P$ from $x$ to $y$ in $G - F$.
        By definition of $F$, the sets $E(P \cap G_\beta)$ are now empty for all $\beta \in J$.
        Similar to before we can apply \cref{lem:efficient} to assume that $P$ is efficient and with the additional property given by the lemma, the sets $E(P \cap G_\beta)$ are still empty for all $\beta \in J$.
        Hence, $P$ is an efficient path from $x$ to $y$ in $G - F$.
        We let
        \[
            I_\alpha \coloneq \set{k \in I}{E(P \cap G_k) \neq \emptyset}
        \]
        be the indices of the fragments that contain at least one edge of $P$.
        We now verify that $I_\alpha$ satisfies the claimed properties.

        The set $I_\alpha$ is disjoint to all $(I_\beta \colon \beta < \alpha)$, since $P$ is edge-disjoint from $F$.
        Since $P$ contains only finitely many edges, $I_\alpha$ is finite.
        Also $x, y \in V(P)$, so $x, y \in V(G(I_\alpha))$.
        To check the last property $\alambda_{G(I_\alpha)}(x, y) = \aleph_0$, we apply \cref{lem:decomposition-connectivity} to $P$ and $x, y \in V(G)$.
        Similar to the previous case the lemma shows $\lambda_{G_{i_m}}(x_m, x_{m + 1}) \geq \minset{\lambda_G(x, y), \aleph_0} = \aleph_0$ for all $m \in [n]$ with the notation as in \cref{lem:decomposition-connectivity}.
        We have $i_m \in I_\alpha$ for all $m \in [n]$ by definition of $I_\alpha$.
        Using that the $\vv{G}_{i_m}$ are well-balanced for all $m \in [n]$, we get $\alambda_{G(I_\alpha)}(x, y) \geq \aleph_0$ by transitivity.
        This concludes the recursion.

        We still need to show that $\alambda_G(x, y) = \lambda_G(x, y) = \lambda$.
        But for this we simply choose a single directed path $P_\alpha$ from $x$ to $y$ in $\vv{G}(I_\alpha)$ for each $\alpha < \lambda$ using property \ref{it:alambda-aleph_0}.
        Since the sets $(I_\alpha \colon \alpha < \lambda)$ are pairwise disjoint, the directed paths $(P_\alpha \colon \alpha < \lambda)$ will be pairwise edge-disjoint.\qedhere
    \end{case}
\end{proof}

We now have all ingredients to prove the main result of this paper:

\thmraylessintro*

\begin{proof}
    Let $\A$ be the class of all rayless graphs. Since $\A$ is closed under subgraphs, \cref{thm:reduction-to-countable-case} implies that in order to find a well-balanced orientation for all rayless graphs, it suffices to find well-balanced orientations for all countable rayless graphs. But every countable rayless graph has a well-balanced orientation by \cref{thm_countablerayless}.
\end{proof}

\bibliographystyle{plain}
\bibliography{references}

\end{document}